\documentclass[a4paper,leqno]{amsart}
\usepackage{amssymb,amsthm,amsmath,mathdots}
\usepackage[alphabetic]{amsrefs}
\usepackage{mathptmx} 
\usepackage{stmaryrd} %for \llbracket and \rrbracket
\usepackage{hyperref}
\usepackage{amsbsy}

\usepackage{lineno}
\newcommand{\bwr}{\boldsymbol{\wr}}

\newcommand{\1}{\mathbf{1}}
\newcommand{\e}{\varepsilon}
\renewcommand{\l}{\lambda}

\usepackage{fullpage,hyperref,mathrsfs,txfonts,epsfig,graphicx,graphics,verbatim}

\newtheorem{thm}{Theorem}

\newtheorem*{mainproblem*}{Problem}

\newtheorem{lem}[thm]{Lemma}
\newtheorem{prop}[thm]{Proposition}
\newtheorem{cor}[thm]{Corollary}

\theoremstyle{definition}

\newtheorem*{defn}{Definition}

\theoremstyle{remark}

\newtheorem{rem}{Remark}

\DeclareMathOperator{\cay}{Cay}

\DeclareMathOperator{\tsp}{tsp}
\DeclareMathOperator{\TSP}{TSP}

\newcommand{\com}{\ensuremath{\mathrm{K}}}

\newcommand{\Ts}{\ensuremath{\mathsf{Ts}}}
\newcommand{\La}{\ensuremath{\mathsf{La}}}

\usepackage{ifthen}

\newcounter{mylisti} \newcounter{mylistii}
\newcounter{nest}
\newcommand{\defaultlabel}{}

\usepackage{accents}

\newcommand{\widethickbar}[1]{\accentset{\rule{.6em}{.6pt}}{#1}}

\newcommand{\psib}{\ensuremath{\widethickbar{\psi}}}

\newcommand{\be}{\ensuremath{\mathbb E}}

\newcommand{\bp}{\ensuremath{\mathbb P}}

\newcommand{\bz}{\ensuremath{\mathbb Z}}

\newcommand{\cP}{\ensuremath{\mathcal P}}

\usepackage{mathrsfs}

\newcommand{\ft}{\ensuremath{\tilde{f}}}

\newcommand{\abs}[1]{\lvert #1\rvert}

\newcommand{\diam}{\operatorname{diam}}

\newcommand{\fin}[1]{{[#1]}^{<\omega}}

\newcommand{\bi}{\ensuremath{\boldsymbol{1}}}

\newcommand{\norm}[1]{\lVert #1\rVert}

\newcommand{\E}{\exists\,}

\newcommand{\varf}{\varphi}

\renewcommand{\geq}{\geqslant}
\renewcommand{\leq}{\leqslant}

\newcommand{\ds}{\displaystyle}

\newcommand{\ie}{\textit{i.e.,}\ }

\newcommand{\cf}{\textit{cf.}\ }

\newcommand{\bE}{\ensuremath{\mathbb E}}

\newcommand{\bN}{\ensuremath{\mathbb N}}

\newcommand{\bR}{\ensuremath{\mathbb R}}

\newcommand{\bZ}{\ensuremath{\mathbb Z}}

\newcommand{\symdif}{\triangle}
\newcommand{\LF}{\mathsf{LF}}

\begin{document}

%\linenumbers
    
\title{Stochastic approximation of lamplighter metrics}

\author{F.~Baudier}
\address{Department of Mathematics, Texas A\&M University, College
  Station, TX 77843, USA}
\email{florent@tamu.edu}

\author{P.~Motakis}
\address{Department of Mathematics, University of Illinois at
  Urbana-Champaign, Urbana, IL 61801, USA}
\email{pmotakis@illinois.edu}

\author{Th.~Schlumprecht}
\address{Department of Mathematics, Texas A\&M University, College
  Station, TX 77843, USA and Faculty of Electrical Engineering, Czech
  Technical University in Prague,  Zikova 4, 166 27, Prague}
\email{schlump@math.tamu.edu}

\author{A.~Zs\'ak}
\address{Peterhouse, Cambridge, CB2 1RD, UK}
\email{a.zsak@dpmms.cam.ac.uk}

\date{\today}

\thanks{F. Baudier was supported by the National Science
  Foundation under Grant Number DMS-1800322. P. Motakis
  was  supported by National Science Foundation under Grant Numbers
  DMS-1600600 and DMS-1912897. Th. Schlumprecht was supported by the National
  Science Foundation under Grant Numbers DMS-1464713 and
  DMS-1711076. A. Zs\'ak was supported by the 2019 Workshop
  in Analysis and Probability at Texas A\&M University.}


\begin{abstract}
We observe that embeddings into random metrics can be fruitfully used to study the $L_1$-embeddability of lamplighter graphs or groups, and more generally lamplighter metric spaces. Once this connection has been established, several new upper bound estimates on the $L_1$-distortion of lamplighter metrics follow from known related estimates about stochastic embeddings into dominating tree-metrics. For instance, every lamplighter metric on a $n$-point metric space embeds bi-Lipschitzly into $L_1$ with distortion $O(\log n)$. 
In particular, for every finite group $G$ the lamplighter group $H = \mathbb{Z}_2\wr G$ bi-Lipschitzly embeds into $L_1$ with distortion $O(\log\log|H|)$.
In the case where the ground space in the lamplighter construction is a graph with some topological restrictions, better distortion estimates can be achieved. Finally, we discuss how a coarse embedding into $L_1$ of the lamplighter group over the $d$-dimensional infinite lattice $\bZ^d$ can be constructed from bi-Lipschitz embeddings of the lamplighter graphs over finite $d$-dimensional grids, and we include a remark on Lipschitz free spaces over finite metric spaces.
\end{abstract}

\maketitle

%\setcounter{tocdepth}{1}
%\tableofcontents
\section{Introduction}
Understanding how a group or a graph, viewed as a geometric object, can be faithfully embedded into certain Banach spaces is a fundamental topic with far reaching applications (see for instance \cite{AIR18}, \cite{BaudierJohnson16}, \cite{Indyk01}, \cite[Chapter 15]{Matousek_book}, \cite{Naor10}, \cite{Naor13}, \cite{NowakYubook}, and \cite{Ostrovskii_book} and the references therein) that is common to geometric group theory and theoretical computer science. The wreath product of two groups, and in particular lamplighter groups, received a lot of attention in geometric group theory as they provide a wealth of groups with interesting algebraic and geometric properties. The main difficulty in studying a lamplighter group metric comes from the fact that it involves a travelling salesman problem on a Cayley graph of the group. Travelling salesman problems are typically $\mathsf{NP}$-hard and thus estimating exactly pairwise distances with respect to a lamplighter group metric becomes increasingly difficult as soon as we drift too far away from groups such as free groups or cyclic groups, for which the travelling salesman problem can be handled adequately.

Coarse embeddings of lamplighter groups over infinite groups which are  finitely generated  have been intensively investigated (see for instance \cite{ADS09}, \cite{ANP09}, \cite{CSV08}, \cite{CTV07}, \cite{csv:12}, \cite{Li10}, \cite{naor-peres:08}, \cite{naor-peres:11}, \cites{StalderValette07}, \cite{Tessera08}, \cite{Tessera09}, \cite{Tessera11}). Bi-Lipschitz embeddings of finite lamplighter groups were considered in \cite{ANV10}, \cite{LNP09}, \cite{JolissaintValette14}, \cite{Tessera12}. The bi-Lipschitz geometry of lamplighter graphs was initiated in \cite{bmsz:19}. We will now focus our attention to the bi-Lipschitz embeddability into the Banach space $L_1$ which is a prime target in theoretical computer science. It was shown in~\cite{naor-peres:08} that finite cyclic lamplighter groups bi-Lipschitzly embed into $L_1$ with distortion bounded above independently of the size of the cyclic group. It follows from this result and general principles that the lamplighter group over the infinite cyclic group bi-Lipschitzly embeds into $L_1$ (an alternative direct proof is also given in \cite{naor-peres:08}). These results were extended to non-superreflexive Banach space targets in \cite{OR18}. That the lamplighter group over a finitely generated free group bi-Lipschitzly embeds into $L_1$, is a by-product of the work of de Cornulier, Stadler and Valette \cite{csv:12} on the equivariant $L_1$-compression of wreath products.  In \cite{bmsz:19}, it was shown that the lamplighter graph over \emph{any} tree bi-Lipschiztly embeds into $L_1$ with distortion at most ~6. In all the $L_1$-embedding results mentioned above, the travelling salesman problems involved have an exact closed formula. However, even on simple planar graphs such as $2$-dimensional grids or lattices, the travelling salesman problem is $\mathsf{NP}$-hard. Naor and Peres \cite{naor-peres:11} showed that lamplighter groups over groups of polynomial growth admit coarse embeddings into $L_1$ that are very close to being bi-Lipschitz. Naor and Peres argument, which uses a multi-scale approximation of the travelling salesman problem that is inspired by Jones' travelling salesman theorem \cite{Jones1990}, will be discussed in Section \ref{sec:coarse}.

In this paper, we propose a new approach to construct bi-Lipschitz embeddings of lamplighter metrics into $L_1$. In the design of approximation algorithms, a classical technique is to reduce a difficult optimization problem on a graph to a corresponding optimization problem on a tree. The optimization problem on a tree is in general much more tractable, and typically its solution can be transferred back to the original graph with some guarantees on the loss incurred in the process. To incorporate this tree-reduction paradigm in geometric approximation algorithms, it would be desirable to embed bi-Lipschiztly a graph into a tree metric keeping the distortion low. Unfortunately, this approach already breaks down for cycles. Indeed, it can be shown (see \cite{RabinovichRaz98} for a topological argument and \cite{Gupta01} for a combinatorial one) that embedding the $n$-cycle into a tree metric incurs distortion $\Omega(n)$ . A simple, but fundamental, observation attributed to Karp \cite{Karp89}, is that if we uniformly randomly delete an edge from the $n$-cycle there is a natural embedding into a simple path of length $n-1$  (an thus into a tree) which does not contract the distances, and expands them by only a factor of $2$ \emph{on average}. Bartal \cite{Bartal96} formalized this powerful idea and introduced the concept of stochastic embeddings into 
dominating tree-metrics (also called random embeddings into dominating trees, embeddings into distribution of trees, embeddings into convex combinations of trees, etc.) and gave a general upper bound on the distortion. The connection with $L_1$-embeddability stems from the simple fact that embeddability into distribution of trees implies embeddability into $L_1$. The (stochastic) distortion of embeddings into random trees has been estimated for several important classes of metric spaces and it will be discussed in Section \ref{sec:app}. 

Indyk and Thaper explained in \cite{IndykThaper03} that a result of Charikar from \cite{Charikar02} implicitly showed that the techniques of Kleinberg and Tardos \cites{KleinbergTardos99,KleinbergTardos02} implied that if a metric space $X$ can be stochastically embedded into dominating tree-metrics with distortion $D$ then the Earthmover distance over $X$ can be embedded into $\ell_1$ with distortion $O(D)$. Our main result shows that a similar phenomenon holds for lamplighter metrics. This new approach combined with the $L_1$-embedding of lamplighter tree metrics from \cite{bmsz:19} allows us to recover the $L_1$-embedding results for finite cyclic lamplighter groups and toric lamplighter groups, with matching distortion. Moreover, $L_1$-distortion upper bounds for lamplighter metrics over new classes of metrics are obtained. In light of the arguable difficulty to understand the planar lamplighter group metric, it seems (at least to us) that a \emph{direct} approach to construct embeddings into $L_1$ for these fairly large classes of metrics (such as outerplanar metrics or finite pathwidth metrics) would be a rather daunting task.

In Section \ref{sec:lamp-tsp} we introduce the general concept of lamplighter metrics over metric spaces and travelling salesman semimetrics, and in Section \ref{sec:lifting} we prove our main result. Applications and limitations of the approach using embeddings into distribution of trees are discussed in Section \ref{sec:app}.

\subsection*{Acknowledgements}
The first author would like to thank Alex Andoni, Tasos Sidiropoulos, and Ilya Razenshteyn for illuminating discussions on the geometry of the Earthmover distance and embeddings into random metrics. 
\section{Lamplighter metrics and travelling salesman semi-metrics}\label{sec:lamp-tsp}
Let $G=(V,E)$ be a graph, which may be finite or infinite, and in
the latter case we do not necessarily assume $G$ to be locally
finite. The \emph{lamplighter graph $\La(G)$ of $G$ }is the graph with
vertex set consisting of all pairs $(A,x)$ where $A$ is a finite
subset of $V$ and $x\in V$. Two vertices $(A,x)$ and $(B,y)$ are
connected by an edge if and only if either $A=B$ and $xy$ is an edge
in $G$, or $x=y$ and $A\symdif B=\{x\}$. The usual interpretation of
$\La(G)$ is as follows. Each vertex of $G$ has a light attached to
it. A vertex $(A,x)$ of $\La(G)$ is the configuration in which $A$ is
the set of lights that are switched on, and a lamplighter is standing
at vertex $x$. The lamplighter can make one of two possible moves: he
can either move to another vertex $y$ along an edge in $G$, or stay at
$x$ and change the status of the light at vertex $x$. Lamplighter
graphs are generalizations of lamplighter groups as follows. Given a
group $\Gamma$ with generating set $S$, the lamplighter graph of the
Cayley graph $\cay(\Gamma,S)$ of $\Gamma$ with respect to $S$ is the
Cayley graph of the lamplighter group of $\Gamma$ with respect to an
appropriate generating set.

When $G$ is connected, it becomes a metric space with the graph
distance $d=d_G$ given by the length of a shortest path between
vertices. In this case $\La(G)$ is also connected and its graph
distance is given by the following formula
(see for instance~\cite{bmsz:19}*{Proposition~1}).
\begin{equation}\label{eq:lamplighter-distance}
  d_{\La(G)}\big((A,x),(B,y)\big)=\tsp_G(x,A\symdif B,y)+\abs{A\symdif
    B}
\end{equation}
where $\tsp_G(x,C,y)$ denotes the solution of a combinatorial travelling salesman
problem in $G$ for $x,y\in G$ and $C$ a finite subset of $G$. More precisely,
$\tsp_G(x,C,y)$ is the least $n\geq 0$ such that there is a sequence
$x_0,x_1,\dots,x_n$ of vertices in $G$ satisfying $x_0=x$, $x_n=y$,
$x_{i-1}x_i\in E(G)$ for $1\leq i\leq n$ and
$C\subset\{x_0,\dots,x_n\}$.

We will be working in the more general framework of lamplighter metrics over metric spaces since it slightly simplifies the exposition of the ensuing arguments, and it is more convenient to formalize results for lamplighter metrics over weighted graphs. In the metric travelling salesman problem there is given a set $V$ together with a symmetric cost function $c(x,y)$, $x,y\in V$, which represents the cost of travelling between $x$ and $y$ and that satisfies the triangle inequality. The problem, which we denote by $\TSP_V(x,A,y)$, of finding the minimal cost $\tsp_V(x,A,y)$ of a simple path in $\com_V$, the complete graph over $V$ (with edge-weight function $c$), from $x$ to $y$ that visits all points in the finite set $A$
is known to be NP-hard. By analogy with the lamplighter graph construction, we make the following definition. In the sequel, $\fin{X}$ denotes the set of all finite subsets of a set $X$.
\begin{defn}
  Let $X$ be a metric space with metric $d_X$. We define the
  \emph{lamplighter metric space $\La(X)$ over $X$ }to be the set
  $\fin{X}\times X$ equipped with the metric
  \begin{equation}
  \begin{aligned}
    \label{eq:metric-lamplighter-distance}
    d_{\La(X)}((A,x),(B,y))&=\tsp_X(x,A\symdif B,y)
    +\abs{A\symdif B}.
    \end{aligned}
  \end{equation}
  We will drop the subscripts $X$ and $\La(X)$ when appropriate.
\end{defn}

The fact that $d_{\La(X)}$ is a metric follows from the fact that it is the sum of a semi-metric $\tau_X$ on $\fin{X}\times X$, defined by \[\tau_X((A,x),(B,y))=\tsp_X(x,A\symdif B,y),\] and the classical symmetric difference metric $d_{\symdif}$ on $\fin{X}$ where $d_{\symdif}(A,B)=|A\symdif B|$.
The fact that this is indeed an extension of the notion of lamplighter graph amounts to proving that $\tsp_X(x,A,y)=\tsp_G(x,A,y)$ for $x,y\in G$ in the case where the metric on $X$ coincides with the graph metric on the vertex set of a connected graph $G$, $\ie d=d_X=d_G$. In this situation, the graph distance $d_{\La(G)}$ on the vertex set of the lamplighter graph $\La(G)$ is precisely the metric $d_{\La(X)}$ of the lamplighter space $\La(X)$ as defined above.

\begin{lem}
Let $G=(V,E)$ be a connected graph and denote by $(X,d)$ the set $V$ equipped with the canonical graph metric. Then the combinatorial travelling salesman problem on $G$ coincides with the metric travelling salesman problem on $(X,d)$, \ie
$\tsp_{X}(x,A,y)=\tsp_G(x,A,y)$.
\end{lem} 

\begin{proof}
Given a walk
  $x_0,x_1,\dots, x_n$ in $\com_X$ from $x=x_0$ to $y=x_n$ visiting
  all vertices in $A$, for each $i=1,\dots,n$ let $w_i$ be a path in
  $G$ from $x_{i-1}$ to $x_i$ of length
  $d(x_{i-1},x_i)$. Concatenating $w_1,\dots,w_n$ yields a walk in $G$
  from $x$ to $y$ of length $\sum_{i=1}^nd(x_{i-1},x_i)$. This shows
  the inequality $\tsp_{X}(x,A,y)\geq\tsp_G(x,A,y)$. For the reverse
  inequality, consider a walk $x_0,x_1,\dots, x_n$ in $G$ from $x=x_0$
  to $y=x_n$ visiting all vertices in $A$. Then $d(x_{i-1},x_i)=1$ for
  all $i$, and hence the total cost $\sum_{i=1}^nd(x_{i-1},x_i)$ is
  the length of the walk. It follows that
  $\tsp_G(x,A,y)\geq\tsp_{X}(x,A,y)$.
\end{proof}

Since our main concern is the embeddability of lamplighter spaces into $L_1$, and since $d_\symdif$ (the symmetric difference metric on
$\fin{X}$) is an $L_1$-metric, we shall focus our attention to the travelling salesman semi-metric $\tau_X$. For a ground metric space $(X,d_X)$, we will denote by $\Ts(X)$ the semi-metric space $(\fin{X}\times X,\tau_X)$. In order to embed $\La(X)$, it is sufficient to embed $\Ts(X)$. This fact will be used repeatedly in the ensuing arguments. 

\begin{lem}
  \label{lem:ts-to-lamplighter}
  Let $(X, d_X)$ be a metric space and
  $f\colon\Ts(X)\to L_1(\mu)$ be a function such that there exists
  a constant $D\ge 1$ with
  \[
\tau_X(u,v)  
  \leq \norm{f(u)-f(v)}_{1}
  \leq D\tau_X(u,v)
  \]
  for all $u,v\in\Ts(X)$. Then there is a measure $\nu$ and
  there is a function $\ft\colon\La(X)\to L_1(\nu)$ satisfying
  \[
d_{\La(X)}(u,v) \leq \norm{\ft(u)-\ft(v)}_{1} \leq
  D\cdot d_{\La(X)}(u,v)
  \]
  for all $u,v\in\La(X)$.
\end{lem}

\begin{proof}
  Define $\ft\colon\La(X)\to L_1(\mu)\oplus_1\ell_1(X)$ by letting
  \[
  \ft(A,x)=\big( f(A,x),\bi_A \big)\ .
  \]
  Then
  \[
  \norm{\ft(A,x)-\ft(B,y)}_1=\norm{f(A,x)-f(B,y)}_1+
\abs{A\symdif B}\ .
  \]
  The result now follows from the assumption on $f$ and from
  Definition~\eqref{eq:metric-lamplighter-distance} of the lamplighter metric.
\end{proof} 

\section{A lifting lemma for travelling salesman semi-metrics}
\label{sec:lifting}
In the remainder of the text a genuine metric or a semi-metric will simply be referred to as a metric, since the results and definitions below only rely on the triangle inequality and the symmetry property of metrics. Recall that $(X,d_X)$ embeds bi-Lipschiztly into $(Y,d_Y)$ if there exist a scaling factor $s>0$, a constant $D\geq 1$, and a map $f\colon X\to Y$ such that 
\begin{equation}
s\cdot d_X(x,y)\leq d_Y(f(x),f(y))\leq D\cdot s\cdot d_X(x,y).
\end{equation}

The parameter $c_Y(X)$, called the $Y$-distortion of $X$ is the infimum of such constant $D$ over all bi-Lipschitz embeddings $f$ as above.

From now on we will only consider finite metric spaces and finite graphs. Let $\mathsf{Y}$ be a class of metric spaces. For $D\geq1$, we say that a metric space $X$ admits a \emph{stochastic $D$-embedding into dominating $\mathsf{Y}$-metrics} if there exist non-negative numbers $p_1, p_2,\dots,p_k$ such that $\sum_{i=1}^k p_i=1$, and maps $f_i\colon X\to (Y_i,d_i)\in\mathsf{Y}$ such that for all $x,y\in X$
\begin{equation}\label{eq:lower}
d_i(f_i(x),f_i(y))\geq d_X(x,y) \textrm{ for all } 1\leq i\leq k
\end{equation}
and 
\begin{equation}\label{eq:upper}
\sum_{i=1}^k p_id_i(f_i(x),f_i(y))\leq Dd_X(x,y).
\end{equation}

Stochastic embeddings into dominating $\mathsf{Y}$-metrics were put forth in the context of dominating tree-metrics by Bartal \cite{Bartal96} following the work of Alon, Karp, Peleg and West \cite{AKPW95}. The definition above has a natural probabilistic interpretation in the sense that a metric space stochastically $D$-embeds into dominating $\mathsf{Y}$-metrics if there is a randomly chosen metric space $Y\in \mathsf{Y}$ and a randomly chosen map $f\colon X\to Y$ such that with probability $1$, $f$ is non-contracting and for all $x,y\in X$, 
\[\bE[d_Y(f(x),f(y))]\leq Dd_X(x,y).\]
Stochastic embeddings into dominating $L_1$-embeddable-metrics turn out to be extremely useful when trying to embed metric spaces into $L_1$ thanks to the following well-known fact.
\begin{lem}\label{lem:stochastic-to-deterministic}
Let $\mathsf{Y}$ be a collection of finite metric spaces and $c_1(\mathsf{Y}):=\sup_{Y\in\mathsf{Y}}c_{L_1}(Y)$. If $(X,d_X)$ admits a stochastic $D$-embedding into dominating $\mathsf{Y}$-metrics then $X$ bi-Lipschitzly embeds into $L_1$ with distortion at most $D\cdot c_1(\mathsf{Y})$.
\end{lem}

\begin{proof}
Let $p_1, p_2,\dots,p_k$, $(Y_1,d_1),\ldots,(Y_k,d_k)\in\mathsf{Y}$ and $f_i\colon X\to (Y_1,d_1),\ldots f_k:X\to (Y_k,d_k)$ witness the fact that $(X,d_X)$ admits a stochastic $D$-embedding into dominating $\mathsf{Y}$-metrics. For $i=1,\ldots,k$ pick an embedding $g_i:(Y_i,d_i)\to L_1$ for which, after eventually rescaling, we may assume that for all $y,z\in Y_i$ we have $d_i(y,z) \leq \|g_i(y) - g_i(z)\| \leq c_id_i(y,z)$. Define $f:(X,d_X)\to(\oplus_{i=1}^kL_1)_1$ given by $f(x) = (p_1g_1(f_1(x)),\ldots,p_kg_k(f_k(x)))$. A direct computation yields that for $x_1,x_2\in X$ we have
\[d_X(x_1,x_2) \leq \|f(x_1) - f(x_2)\| \leq D\big(\max_{1\leq i\leq k}c_i\big)d_X(x_1,x_2).\]
Taking an infimum over all choices of $g_1,\ldots,g_k$ (which is in fact  achievable by an elementary argument since $X$ is finite) yields the desired estimate.
\end{proof}

The travelling salesman (semi-)metric associated with a finite metric space $(X,d_X)$ can be conveniently reformulated as follows
\begin{equation}\label{eq:ts-metric}
\tau_X((A,s),(B,t))=\displaystyle\min_{P\in \cP(s, A\symdif B, t)} \sum_{(u,v)\in P}d_X(u,v),
\end{equation}
where the minimum is taken over all paths $P$ in $\cP(s, A\symdif B, t)$, the set of all simple paths in the complete graph over $(A\symdif B)\cup\{s,t\}$ connecting $s$ to $t$ and visiting all the vertices in $A\symdif B$ (exactly once). A path that achieves the minimum will be referred to as a \emph{minimizing path}, and $(u,v)\in P$ denotes an edge in $P$.

The following crucial fact, inspired by a similar fact which is pivotal in the study of transportation cost metrics and the Earthmover distance, shows that an embedding into random trees for the ground space can be lifted to the travelling salesman spaces over the corresponding metrics.
  
\begin{lem}\label{lem:lifting}
If $X$ admits a stochastic $D$-embedding into dominating $\mathsf{Y}$-metrics, then $\Ts(X)$ admits a stochastic $D$-embedding into dominating $\{\Ts(Y)\colon Y\in\mathsf{Y}\}$-metrics. 
\end{lem}

\begin{proof}
Assume that $X$ admits a stochastic $D$-embedding into dominating $\mathsf{Y}$-metrics and let $(p_i)_{i=1}^k$ and $f_i\colon X\to (Y_i,d_i)$ as in the definition above. We will also denote $f_i$ the set-function that maps $A\subseteq X$ to $f_i(A)\subseteq Y_i$, and we will show that for all $(A,s),(B,t)$ in $\Ts(X)$
\begin{equation}
\tau_{Y_i}((f_i(A),f_i(s)),(f_i(B),f_i(t)))\geq \tau_X((A,s),(B,t))
\end{equation}
and 
\begin{equation}
\sum_{i=1}^k p_i\tau_{Y_i}((f_i(A),f_i(s)),(f_i(B),t))\leq D\cdot\tau_{X}((A,s),(B,t)).
\end{equation}

Since $f_i$ is injective, $f_i(A)\symdif f_i(B)\cup\{f_i(s),f_i(t)\}=f_i(A\symdif B\cup\{s,t\})$ and $f_i$ is a bijection between $A\symdif B\cup\{s,t\}$ and $f_i(A\symdif B\cup\{s,t\})$. Observe that a bijection between two complete graphs of equal size is automatically a graph isomorphism, and hence for any path in $\cP(s, A\symdif B,t)$ there is a corresponding path in $\cP(f_i(s), f_i(A\symdif B), f_i(t))$ and vice-versa. 

Let $\tilde{P}_i$ be a minimizing path in $\cP(f_i(s), f_i(A\symdif B), f_i(t))$ whose cost is $c(\tilde{P}_i):=\sum_{(u,v)\in \tilde{P}_i}d_{Y_i}(u,v).$ 
Let $P_i$ be the corresponding path in $\cP(s, A\symdif B,t\}$, i.e. $(x,y)\in P_i$ if and only if $(f_i(x),f_i(y))\in \tilde{P}_i$. Then,

\begin{align*}
c(\tilde{P}_i):=\sum_{(u,v)\in \tilde{P}_i}d_{Y_i}(u,v)=\sum_{(x,y)\in P_i}d_{Y_i}(f_i(x),f_i(y))\stackrel{\eqref{eq:lower}}{\geq} \sum_{(x,y)\in P_i} d_{X}(x,y)\stackrel{\eqref{eq:ts-metric}}{\geq} \tau_X((A,s),(B,t)),
\end{align*}
and hence $\tau_{Y_i}((f_i(A),f_i(s)),(f_i(B),f_i(t)))=c(\tilde{P}_i)\geq \tau_X((A,s),(B,t))$.

Now, let $P$ be a minimizing path in $\cP(s, A\symdif B,t\}$,  and let  $\tilde{P}_i$ be the corresponding path in $\cP(f_i(s), f_i(A\symdif B), f_i(t))$. Taking a convex combination with the weights $(p_i)_i$ we have
\begin{align*}
\sum_{i=1}^k p_i c(\tilde{P}_i)&= \sum_{i=1}^k p_i \sum_{(u,v)\in \tilde{P}_i}d_{Y_i}(u,v)=\sum_{i=1}^k p_i \sum_{(x,y)\in \tilde P_i}d_{Y_i}(f_i(x),f_i(y)) \textrm{ (by definition of $\tilde{P}_i$)}\\
					&=\sum_{(x,y)\in P} \sum_{i=1}^k p_i d_{Y_i}(f_i(x),f_i(y))\stackrel{\eqref{eq:upper}}{\leq} \sum_{(x,y)\in P} D\cdot d_{X}(x,y)\\
					&= D\cdot \tau_{X}((A,s),(B,t)) \textrm{ (by minimality of $P$)},
\end{align*}
and thus by taking minima on the left-hand side we have 
$$\sum_{i=1}^k p_i\tau_{Y_i}((f_i(A),f_i(s),(f_i(B),f_i(t))\leq D\cdot\tau_{X}((A,s),(B,t)).$$
\end{proof}

Combining Lemma \ref{lem:ts-to-lamplighter}, Lemma \ref{lem:stochastic-to-deterministic}, and Lemma \ref{lem:lifting} we obtain the following theorem. 

\begin{thm}\label{thm:main}
If $(X,d_X)$ admits a stochastic $D$-embedding into dominating $\mathsf{Y}$-metrics then $\La(X)$ embeds into $L_1$ with distortion at most $D\cdot \sup_{Y\in\mathsf{Y}}c_1(\Ts(Y))$.
\end{thm}

The metrics appearing when embedding into random trees are metrics supported on \emph{weighted} trees. Recall that a metric $d_w$ is \emph{supported} on a graph $G=(V,E)$ if the metric is the shortest path metric on $G$ with respect to some weight function $w\colon E\to [0,\infty)$. We always assume that the weight function is reduced, i.e. $d_w(x,y)=w((x,y))$ whenever $(x,y)\in E$. We will simply say \emph{tree-metric} for a metric that is supported on a tree. We will need the following proposition which is a weighted version of a result that can be shown in exactly the same way as the unweighted version  \cite[Theorem A and subsequent paragraph]{bmsz:19}, but also  deduced from it by subdividing edges to approximate a finite weighted graph by one with constant weights. 

\begin{prop}\label{prop:tree-embedding}
Let $d_w$ be a metric supported on a tree $T$. Then $\Ts((T,d_w))$ admits a bi-Lipschitz embedding into $L_1$ with distortion at most~6. 
\end{prop}

For the benefit of the reader we give an explicit description of the bi-Lipschitz embedding in the above Proposition while omitting the proof of its properties. We recall the necessary pieces of notation from \cite{bmsz:19}. For vertices $x$ and $y$ in a tree $T$ we denote by $[x,y]$ the set of edges on the unique path from $x$ to $y$, and for $A\subset T$ we let
$[x,A]=\bigcup_{a\in A} [x,a]$. The weighted analogue of \cite[Theorem 3]{bmsz:19} can be stated as follows:
for $x,y\in T$ and a finite $A\subset T$,
  \[
  \tsp_{T}(x,A,y)=2\sum_{e\in [x,A]\setminus [x,y]}w_e+\sum_{e\in[x,y]}w_e.
  \]

  For $C\subset T$ let $[C]=\bigcup _{x,y\in C}[x,y]$ be the minimal set of
edges needed to travel between different vertices of $C$. Define
\[
I=\big\{ (e,C):\, e\in E(T),\ \emptyset\neq C\subset T,\ C\text{
  finite},\ e\notin [C]\big\}\ .
\]
For $A\subset T$, $x\in T$ and $e\in E(T)$, let
$A_{x,e}=\big\{a\in A:\, e\in[x,a]\big\}$.

Define
  $f\colon \Ts(T)\to \ell_1(I)$ by
  \[
  f(A,x)=(f(A,x)_i)_{i\in I}\ ,
  \]
  where $f(A,x)_i=w_e$ if there exists  $e\in E(T)$ such that $A_{x,e}\neq\emptyset$ and $i=(e,A_{x,e})$, and $f(A,x)_i=0$ otherwise. An argument similar to the one given in the proof of Lemma~9 in \cite{bmsz:19}, will show that
\begin{equation}\label{eq:inequalities}
\sum_{e\in[x,A\symdif B]\setminus [x,y]}w_e \leq \|
  f(A,x)-f(B,y)\|_1 \leq 2\sum_{e\in[x,A\symdif B]\setminus[x,y]} w_e+
  2\sum_{e\in[x,y]}w_e\ .
\end{equation}
Fix a root $r\in T$ and define
  $g\colon \Ts(T)\to \ell_1(I)\oplus \ell_1(E(T))$ by $g(A,x)=\big(f(A,x),(w_{e})_{e\in[r,x]}\big)$.
It can then be verified, as done in \cite{bmsz:19}, that $g$ is a bi-Lipschitz embedding with distortion at most~6. The next corollary follows immediately from Theorem \ref{thm:main} and Proposition \ref{prop:tree-embedding}.

\begin{cor}\label{cor:main}
If $(X,d_X)$ admits a stochastic $D$-embedding into dominating tree-metrics then $\La(X)$ embeds into $L_1$ with distortion at most $6D$.
\end{cor}

%\begin{rem}
%The distortion in Corollary \ref{cor:main} is at most $12D$. This upper bound could be improved to $6D$, but this inessential improvement would have to be done at the expense of clarity. Instead of using Lemma \ref{lem:ts-to-lamplighter}, Lemma \ref{lem:stochastic-to-deterministic}, and Lemma \ref{lem:lifting} as black-boxes, we could write a slightly different explicit embedding (similar to what is done in Theorem~10 in \cite{bmsz:19}) which uses the inequalities \eqref{eq:inequalities} in a slightly more optimal way.
%\end{rem}

\section{Applications and limitations}
\label{sec:app}

We now describe upper bounds on the $L_1$-distortion of lamplighter metrics that can be obtained from Corollary \ref{cor:main} in combination with upper bounds on the distortion to embed the ground spaces into random trees. We start with a general upper bound valid for all lamplighter metrics. Bartal's original upper bound \cite{Bartal96} on the distortion to stochastically embed a $n$-point metric space into dominating tree-metrics was $O(\log^2 n)$. It was improved to $O(\log(n)\log\log n)$ in \cite{Bartal99} and \cite{CCGGP99}. Fakcharoenphol, Rao, and Talwar \cite{FRT04} obtained the tight $O(\log n)$ upper bound. This bound is optimal since expander graphs \cite{LLR95}, or Khot-Naor spaces \cite{KhotNaor06}, cannot embed with a better distortion into $L_1$.

\begin{cor}\label{cor:general}
Let $(X,d_X)$ be a finite metric space with $|X|=n$. Then $\La(X)$ bi-Lipschitzly embeds into $L_1$ with distortion $O(\log n)=O(\log\log |\La(X)|)$.
\end{cor}

Note that since $\La(X)$ has $2^{n}n$ vertices the upper bound above is an exponential improvement over Bourgain's upper bound for arbitrary finite metric spaces. The upper bound in Corollary \ref{cor:general}
is tight since $\La(X)$ contains an isometric copy of $X$ and thus $c_1(\La(X))\geq c_1(X)$, and $c_1(X)=\Omega(\log |X|)$ for expander graphs, or Khot-Naor spaces.

When the ground metric is supported on a graph with some topological constraints, improved upper bounds can be obtained in some cases. A graph is said to be \emph{outerplanar} (or $1$-outerplanar) if it is planar (i.e. it can be drawn on the plane without edges crossing) with an embedding in the plane such that all its vertices lie on the unbounded face. For $k\geq 2$, it is $k$-outerplanar if it is planar and by removing the vertices on the unbounded face we obtain a $(k-1)$-outerplanar graph. The $k\times n$ grid is an example of a $k$-outerplanar graph. The following corollary follows from the distortion of the stochastic embeddings by Gupta, Newman, Rabinovich and Sinclair in \cite{GNRS04} for $k=1$, and from Chekuri et al. in \cite{CGNRS06} for $k\geq 2$, of $k$-outerplanar graphs into $L_1$.

\begin{cor}\label{cor:outerplanar}
Let $k\geq 1$. Let $d_w$ be a metric supported on a finite $k$-outerplanar graph $G$. Then the lamplighter space
$\La((G,d_w))$ bi-Lipschitzly embeds into $L_1$ with distortion $O_k(1)$.
\end{cor}

Note that for $k=1$ the distortion is at most~8 (see \cite[paragraph after the proof of Lemma 5.3]{GNRS04}), while for $k\geq 2$ the dependance is exponential in $k$ (see \cite[Theorem 1.1]{CGNRS06}).

The next corollary follows from \cite[Theorem 5.5]{GNRS04} where a probabilistic embedding into trees was devised for finite graphs and whose distortion depends on the graph's Euler number, namely $\chi(G)=|E|-|V|+1$. 
\begin{cor}\label{cor:Euler}
Let $G$ be a finite graph with Euler characteristic $\chi(G)$. Then the lamplighter graph metric on $G$ bi-Lipschitzly embeds into $L_1$ with distortion $O(\log(\chi(G))$.
\end{cor}

Lee and Sidiropoulos \cite{LeeSidi13} proved that every metric supported on a finite graph with pathwidth $k$ admits a stochastic embedding into dominating tree-metrics with distortion $O_k(1)$. Since it is not needed in the sequel, we will not recall any of the numerous equivalent definitions of the pathwidth of a graph, and we will simply say that it is the minimum width over all path decompositions of $G$ (these are tree decompositions of a special type). The class of pathwidth-$1$ graphs are special type of trees. The $k$-th level diamond graphs are examples of graphs with pathwidth $O(k)$.

\begin{cor}\label{cor:pathwidth}
Let $k\geq 2$ and $d_w$ be a metric supported on a finite graph with pathwidth $k$. Then
$\La((G,d_w))$ bi-Lipschitzly embeds into $L_1$ with distortion $O_k(1)$.
\end{cor}

Since by a result of Robertson and Seymour \cite{RobertsonSeymour83}, the class of graphs which do not have a tree $T$ as a minor have finite pathwidth, we also have:

\begin{cor}\label{cor:tree-excluded}
Let $T$ be a tree and let $d_w$ be any metric supported on a finite graph which does not have $T$ as a minor. Then
$\La((G,d_w))$ bi-Lipschitzly embeds into $L_1$ with distortion $O(1)$.
\end{cor}
 
 Approximating arbitrary metrics by tree-metrics is a powerful tool as can be seen from the previous discussion. It has already been noticed above that the general upper bound from Corollary \ref{cor:general} is incompressible. However, improvements are achievable for restricted classes of metrics as shown by Corollary \ref{cor:outerplanar}, Corollary \ref{cor:Euler}, Corollary \ref{cor:pathwidth} and Corollary \ref{cor:tree-excluded}. Unfortunately, some super-constant upper bounds are also tight for classical graphs with some severe topological restrictions. For instance, the lamplighter metric over the $k$-th level diamond graph embeds into $L_1$ with distortion $O(k)$ by Corollary \ref{cor:general} (Corollary \ref{cor:pathwidth} would guarantee an upper bound no better than  $(4k)^{k^3+1}$, since it builds uppon  \cite[Theorem 1.4]{LeeSidi13}).
 We could hope, since the diamond graphs are planar graphs, and moreover series-parallel graphs, that embeddings into random trees with significantly much lower distortion could be constructed. This hope was ruled out by Gupta et al. \cite{GNRS04}. Indeed, they showed that any embedding into random trees of the $k$-th level diamond graph will suffer a $\Omega(k)$ distortion. Similarly, it was shown by Konjevod, Ravi and Salman \cite{KRS01} that the $n\times n$ grid cannot be embedded into random trees with distortion better than $\Omega(\log n)$.
 This latter result (which extends a previous result of Alon et al. \cite{AKPW95} where the random trees are spanning trees) rules out the possibility of using stochastic embeddings into dominating tree-metrics to answer positively the Naor-Peres planar lamplighter problem.

 \section{Coarse embeddings of lamplighter graphs over $d$-dimensional lattices}
 \label{sec:coarse}
 A finitely generated group $H$ has polynomial growth $\beta\geq 1$ if there are constants $a,b>0$ such that for all $r\geq1$
\begin{equation*}\label{eq:polygrowth}
ar^\beta\leq |B_H(e_H,r)|\leq br^\beta,
\end{equation*} 
where  $B_H(x,r)\coloneqq \left\{y\in H:\
d_H(x,y)\leq r\right\}$, $d_H(\cdot,\cdot)$ the distance with respect to the Cayley graph in $H$,  and $e_H$ is the identity element of $H$. Naor and Peres showed \cite[Proof of Theorem 3.1]{naor-peres:11} that if $\Gamma$ is a finitely generated group with polynomial growth at least quadratic then for every $\varepsilon\in(0,1)$ there exist a constant $C:=C(\Gamma,\varepsilon)>0$ and a coarse embedding  $f\colon \bZ_2\bwr \Gamma\to L_1$ such that for all $u,v\in \bZ_2\bwr\Gamma$, 
\begin{equation}\label{eq:compression}
d_{\bZ_2\bwr\Gamma}(u,v)^{1-\varepsilon}\leq \|f(u)-f(v)\|_1\leq Cd_{\bZ_2\bwr\Gamma}(u,v). 
\end{equation}
In the language of geometric group theory, it follows from inequality \eqref{eq:compression} that $\alpha_1(\bZ_2\bwr\Gamma)=1$, where the numerical parameter $\alpha_1(\Gamma)$, introduced by Guentner and Kaminker \cite{GK04}, denotes the $L_1$-compression of the finitely generated group $\Gamma$. In the same paper, Naor and Peres raised the following problem.
\begin{mainproblem*}[\cite{naor-peres:11}*{Question 10.1}]\label{P:A}
Does the planar lamplighter group $\bZ_2\wr\bZ^2$ embed bi-Lipschitzly into $L_1$?
\end{mainproblem*}

Thus, Problem \ref{P:A} asks whether we can find a map that satisfies inequality \eqref{eq:compression} with $\varepsilon=0$. 

A careful analysis of the proof in \cite{naor-peres:11}, reveals that $C(\Gamma,\e)=O(\e^{-c})$ for some constant $c>0$, and this particular feature of the embedding has some important consequences\footnote{We thank Assaf Naor for pointing this out to us.}. By setting $\e=1/\log n$ (where $\log (x)=\log_2(x)$), and observing that $n^{\frac{1}{\log n}}=2$, it is thus possible to construct maps $f_n\colon \bZ_2\bwr \Gamma\to L_1$ such that, 
\begin{equation}\label{eq:diam}
\frac{d_{\bZ_2\bwr \Gamma}(u,v)}{\log^c n}\leq
\|f_{n}(u)-f_{n}(v)\|_1\lesssim d_{\bZ_2\bwr \Gamma}(u,v),
\end{equation}
whenever $d_{\bZ_2\bwr \Gamma}(u,v)\leq n$.

Now, observing that $\bZ_n^d$ embeds bi-Lipschitzly into $\llbracket n \rrbracket^{2d}\subset \bZ^{2d}$ with distortion at most~$7$ (see for instance \cite[Lemma 6.12]{MendelNaor08} for even more precise estimates) it follows that $\bZ_2\bwr\bZ_n^d$ bi-Lipschitzly embeds into $\bZ_2\bwr\bZ^{2d}$ with distortion at most~18 (see \cite[Section 5]{bmsz:19}). Finally, the diameter of $\bZ_2\bwr\bZ_n^d$ being $O(n^d)$, it follows from  \eqref{eq:diam} for $\Gamma=\bZ^d$ that $c_1(\bZ_2\bwr\bZ_n^d)=O_d(\log^cn)$. For an arbitrary group $\Gamma$ with polynomial growth, we can take $c\leq 3$ (the constant in Assouad's snowflaking embedding theorem plays a role here), and in the case where $\Gamma$ embeds bi-Lipschiztly into $L_p$ for some $p\in(1,2]$ (in particular $\bZ^d$), this can be improved to $c\leq 2$.

\begin{rem}
In \cite{JolissaintValette14} it was shown that $c_2(\bZ_2\wr\bZ_n^d)=\Omega\Big(
n/\sqrt{\log(n)}\Big)$ for $d=2$, and $c_2(\bZ_2\wr\bZ_n^d)=\Omega(n^{\frac{d}{2}})$ for $d\geq 3$. For $d\geq 2$, by \cite[Theorem 3.1]{naor-peres:11}, $\bZ_2\wr\bZ^d$ coarsely embeds into $L_2$ with compression $\frac 12$, and the compression exponent is attained in this case (cf \cite{naor-peres:11}*{Remark 3.2}). It thus follows that $c_2(\bZ_2\wr\bZ_n^d)=O(n^{\frac{d}{2}})$. Therefore,
\begin{equation}
\Omega\Big(\frac{n}{\sqrt{\log{n}}}\Big)=c_2(\bZ_2\wr\bZ_n^2)=O(n),
\end{equation}
and for $d\geq 3$
\begin{equation}
c_2(\bZ_2\wr\bZ_n^d)=\Theta(n^{\frac{d}{2}}).
\end{equation}
It would be interesting to understand the right order of magnitude for $c_2(\bZ_2\wr\bZ_n^2)$.
It is worth mentioning that the asymptotics for $c_p(\bZ_2\wr\bZ_n)$ are known. Indeed, $c_p(\bZ_2\wr\bZ_n)=\Theta((\log n)^{\min\{1/2,1/p\}})$ for $p\in(1,\infty)$ (see \cite{LPP96} or \cite[Proposition 3.10]{LNP09} for the lower bounds, and for the upper bounds \cite{ANV10} (for $p=2$) and \cite{Tessera12} (for $p\geq 2$)).
\end{rem}

The existence of maps as in \eqref{eq:diam} can also be used to improve the polynomial compression function from \eqref{eq:compression} to a super-polynomial compression function. Define $\Phi\colon \bZ_2\bwr \Gamma\to (\oplus_{k=1}^\infty L_1)_1$ by $\Phi(u)=\sum_{k=1}^\infty f_{2^{2^{k}}}(u)/k^2$. Then it is easy to check that $\Phi$ is Lipschitz and $\|\Phi(u)-\Phi(v)\|_1\gtrsim \rho(d_{\bZ_2\bwr \Gamma}(u,v))$ where $\rho(t)=\frac{t}{\log^c(t)(\log\log t)^2}$. This auto-improvement phenomenon of the embedding in \eqref{eq:compression} that is due to the polynomial behavior in $\e$ of the constant $C(\Gamma,\e)$ is quite remarkable and surprising. We will give an alternate way, which relies on the distortion estimates for lamplighter graphs over finite $d$-dimensional grids from Corollary \ref{cor:general}, to obtain the following super-polynomial compression function for $\bZ_2\bwr\bZ^d$.
 
\begin{thm}
  \label{thm:coarse-embedding}
  For all $d\geq 2$ there exist $C:=C(d)>0$ and
  $h\colon\bZ_2\bwr\bZ^d\to L_1$ such that for all $u,v\in\bZ_2\bwr\bZ^d$, 
  \[
  \rho\big(\tau(u,v)\big)\leq \norm{h(u)-h(v)}_{L_1}
  \leq C\tau(u,v),
  \]
  where $\rho(t)=\frac{t}{\log(t)(\log\log t)^2}.$
\end{thm}

It is possible in some situations to exploit distortion estimates to obtain compression estimates (\cf for instance \cite{Austin2011}). Using a barycentric gluing argument, we could get some partial, but far from satisfactory, information. The principle in \cite{ANT2013} which relates compression estimates to distortion estimates for the Heisenberg group, uses crucially the fact that the Heisenberg group has polynomial growth, and does not seem to apply to groups with exponential growth. We construct a ``folding map" that allows us to extend maps on finite $d$-dimensional grids to the whole $d$-dimensional lattice in a way that pushes the distortion bounds for $\La(\bZ_n^d)$ to a compression bound for $\La(\bZ_n^d)$. We first lay the ground for the proof
by first introducing a fair amount of notation. We now fix $d\geq 2$ in the sequel.
Define $\llbracket a,b\rrbracket=\{k\in \bN\colon a\leq k\leq b\}$. We will simply write $\llbracket n\rrbracket$ for $\llbracket 0, n-1\rrbracket=\{0,1,\dots,n-1\}$ and consider the finite $d$-dimensional grid $\llbracket n \rrbracket^d$
which is clearly an isometric subset of $\bZ^d$. For every $s\in\bZ^d$ we let
\[
C_s=n\cdot s+\llbracket n \rrbracket^d
\]
which is a subset of $\bZ^d$ that is naturally isometric
to $\llbracket n\rrbracket^d$. Note that $\llbracket n \rrbracket^d=C_s$ for $s=(0,0,\cdots,0)$, and that $\{C_s\colon s\in \bZ^d\}$ is a partition of $\bZ^d$ into cubic cells isometric to $\llbracket n \rrbracket^d$.

We next set $C^v_s=v+C_s$ for $v\in \llbracket n \rrbracket^d$ and $s\in\bZ^d$. Note that
each of the sets $C^v_s$ is isometric to $\llbracket n\rrbracket^d$, and that for all $v\in \llbracket n\rrbracket^d$, $\{C_s^v\colon s\in \bZ^d\}$ is also a partition of $\bZ^d$, which is the translate by the vector $v$ of the partition $\{C_s\colon s\in \bZ^d\}$.
Note that for any
$v\in \llbracket n\rrbracket^d$, $s\in\bZ^d$ and $A\cup\{x,y\}\subset C_s^v$, the value of
$\tsp(x,A,y)$ is the same whether it is calculated in $\bZ^d$,
or in $C^v_s$. This observation will be used
frequently and implicitly in the sequel.

To prove Theorem \ref{thm:coarse-embedding} we consider the following ``folding map" $\varf_n\colon\bZ\to\llbracket n\rrbracket$
defined by
\[
\varf_n(in+j)=%
\begin{cases}
  j & \text{if $i$ is even}\\
  n-j-1& \text{if $i$ is odd}
\end{cases}
\]
for all $i\in\bz$ and $j\in\llbracket n\rrbracket$. This is $1$-Lipschitz and is
isometric on $in+\llbracket n\rrbracket$ for all $i\in\bz$. It follows that the product map
$\psi_n=\varf_n\otimes\varf_n\otimes\cdots\otimes\varf_n\colon\bZ^d\to\llbracket n\rrbracket^d$ is
$1$-Lipschitz, and for any $s=(s_1,s_2,\dots,s_d)\in\bZ^d$ it is isometric on
$(ns_1+\llbracket n\rrbracket)\times\cdots\times(ns_d+\llbracket n\rrbracket)=n\cdot s+\llbracket n\rrbracket^d=C_s$

For each $v\in \llbracket n\rrbracket^d$ define $\psi^v_n\colon\bZ^d\to\llbracket n\rrbracket^d$ by
$\psi^v_n(x)=\psi_n(x-v)$. Then $\psi^v_n$ is $1$-Lipschitz and isometric on
$C^v_s$ for all $s\in\bZ^d$. A Lipschitz map between metric
spaces induces in a natural way a map between the corresponding
lamplighter spaces (the special case of graphs is discussed
in~\cite{bmsz:19}*{Section~5}). However, this process does not
sufficiently reflect the property of $\psi^v_n$ being isometric on the
sets $C^v_s$. For this reason we proceed differently. Let us write
$A+B$ for $A\symdif B$. Then for sets $A_1,\dots,A_k$ we have 
$\sum_{i=1}^kA_i=A_1\symdif A_2\symdif\dots\symdif A_k$. For 
$v\in \llbracket n\rrbracket^d$, $s\in \bZ^d$   and for a finite subset $F\subset\bZ^d$, let $F^v_s=F\cap C^v_s$,  and define
\[
\sigma^v_n(F)=\sum_{s\in\bZ^d} \psi^v_n(F^v_s)\subset\llbracket n\rrbracket^d.
\]
Note that since $F$ is finite, the sum above has only finitely many non-empty terms.

\begin{lem}
  \label{lem:folding-induced}
  For each $v\in \llbracket n\rrbracket^d$, define
\begin{eqnarray*}
  \psib^v_n\colon\Ts(\bZ^d)&\to& \Ts(\llbracket n\rrbracket^d)\\
  (A,x) &\mapsto& (\sigma^v_n(A),\psi^v_n(x))\ .
\end{eqnarray*}
The map $\psib^v_n$ is
  $1$-Lipschitz. Moreover, given $(A,x),(B,y)\in\Ts(\bZ^d)$, if
  $(A\symdif B)\cup\{x,y\}\subset C^v_s$ for some $s\in\bZ^d$,
  then
  \[
  \tau(\psib^v_n(A,x),\psib^v_n(B,y))=\tau((A,x),(B,y))\ .
  \]
\end{lem}

\begin{proof}
  Fix $(A,x),(B,y)\in\Ts(\bZ^2)$. Given a finite subset $F\subset \bZ^d$, if we let
$F_s=F\cap C_s$ for $s\in\bZ^d$, then $F^v_s=v+\big(F-v\big)_s$ for
  $v\in \llbracket n\rrbracket^d$ and $s\in\bZ^d$. Since translations are
  isometries, and hence preserve $\tau$, without loss of
  generality we can assume $v=0$. When $v=0$ we simplify the notation to $\sigma_n$, and $\psib_n$. Observe that
  \[
  \sigma_n(A)\symdif\sigma_n(B)=\sum_{s\in\bZ^d} \psi_n(A_s) +
  \sum_{s\in\bZ^d} \psi_n(B_s) = \sum_{s\in\bZ^d}
  \psi_n(A_s\symdif B_s)=\sigma_n(A\symdif B)\ ,
  \]
  and hence $\sigma_n(A)\symdif\sigma_n(B)\subset \psi_n(A\symdif B)$.
  It follows that
  \[
  \tau(\psib_n(A,x),\psib_n(B,y))=
  \tsp(\psi_n(x),\sigma_n(A)\symdif\sigma_n(B),\psi_n(y))\leq
  \tsp(\psi_n(x),\psi_n(A\symdif B),\psi_n(y))\ .
  \]
  So in order to show that $\psib_n$ is $1$-Lipschitz, it remains to
  verify that
  \begin{equation}
    \label{eq:folding-on-tsp}
    \tsp(\psi_n(x),\psi_n(C),\psi_n(y))\leq\tsp(x,C,y)
  \end{equation}
  for any $x,y\in\bZ^d$ and any finite $C\subset\bZ^d$.
  Let $\ell=\tsp(x,C,y)$ and $x_0,x_1,\dots,x_\ell$ be a walk in the
  graph $\bZ^d$ from $x=x_0$ to $y=x_\ell$ visiting all vertices
  in $C$. Since $\psi_n$ is $1$-Lipschitz, it follows that for every
  $i=1,\dots,\ell$, the vertices $\psi_n(x_{i-1})$ and $\psi_n(x_i)$ in
  $\llbracket n\rrbracket^d$ are either equal or adjacent. Deleting successive
  repetitions of vertices yields a walk of length at most $\ell$ in
  $\llbracket n\rrbracket^d$ from $\psi_n(x)=\psi_n(x_0)$ to $\psi_n(y)=\psi_n(x_\ell)$
  visiting all vertices of $\psi_n(C)$. This
  shows~\eqref{eq:folding-on-tsp}.

  If we assume that $(A\symdif B)\cup\{x,y\}\subset C_s$ for some $s\in\bZ^d$, then
  \[
  \sigma_n(A)\symdif\sigma_n(B)=
  \psi_n(A)\symdif\psi_n(B)= \psi_n(A\symdif B),
  \]
  and so we need to show that $ \tsp(\psi_n(x),\psi_n(A\symdif B),\psi_n(y))=\tsp(x, A\symdif B,y)$. This
  follows since $\psi_n$ is an isometry between $C_s$ and $\psi_n(C_s)$.
  \end{proof}

We are now in position to prove Theorem  \ref{thm:coarse-embedding}.

\begin{proof}[Proof of Theorem  \ref{thm:coarse-embedding}]
For every $k\geq 1$, let $n_k=2^{2^{k}}$.  By Corollary \ref{cor:general}, there exists a constant $K:=K(d)>0
$ such that for all $k\geq 1$, there exists a map $f_{n_k}\colon \Ts(\llbracket n_k\rrbracket^d)\to L_1$ such that for all $(A,x),(B,y)\in \Ts(\llbracket n_k\rrbracket^d)$
\[
\frac{\tau((A,x),(B,y))}{K\cdot 2^k}\leq\|f_{n_k}(A,x)-f_{n_k}(B,y)\|_1\leq \tau((A,x),(B,y)).
\]

For $k\geq 1$, let $\nu_k$ denote the uniform probability measure on $\llbracket n_k\rrbracket^d$ and define
\begin{eqnarray*}
g_{k}\colon\Ts(\bZ^d)&\to& L_1(\nu_k,L_1)\\
  (A,x) &\mapsto& \big(v\in \llbracket n_k\rrbracket^d\mapsto f_{n_k}\circ\psib^v_{n_k}(A,x)).
\end{eqnarray*}
and 
\begin{eqnarray*}
h\colon\Ts(\bZ^d)&\to& \bigoplus_{k=1}^\infty L_1(\nu_k,L_1)\\
  (A,x) &\mapsto& \sum_{k=1}^\infty \frac{g_{k}(A,x)}{k^2}.
\end{eqnarray*}

Since $\psib_{n_k}^v$ and $f_{n_k}$ are $1$-Lipschitz, so is $ f_{n_k}\circ\psib^v_{n_k}$ for
  every $v\in\llbracket n_k\rrbracket^d$. It follows that $g_{k}$ is also $1$-Lipschitz and hence $h$ is Lipschitz with Lipschitz constant at most $\sum_{k=1}^\infty \frac{1}{k^2}$. It
  remains to show the lower bound.

  Fix $(A,x),(B,y)\in\Ts(\bZ^d)$, and set
  $D=(A\symdif B)\cup\{x,y\}$. Choose $k\geq 1$ so that
  $2^{2^{k-1}}< \tau((A,x),(B,y))\leq 2^{2^k}$. Then $m=\diam(D)\leq\tsp(x, A\symdif B, y)\leq 2^{2^k}$, and
  thus the set
  $W=\{v\in \llbracket n_{k+1}\rrbracket^d:\,\E t\in\bZ^d\ D\subset C^v_t\}$ has size at least
  $(n_{k+1}-m)^d$. 
  It follows that
  \begin{align*}
    \norm{h(A,x)-h(B,y)}_{1} &\geq \frac{\norm{g_{k+1}(A,x)-g_{k+1}(B,y)}_{1}}{(k+1)^2}=\frac{1}{(k+1)^2}\be_{v_k}\bigl( \norm{f_{n_{k+1}}\circ\psib^v_{n_{k+1}}(A,x)-f_{n_{k+1}}\circ\psib^v_{n_{k+1}}(B,y)}_{1}\bigr) \\
					&\geq\frac{1}{K(k+1)^2\cdot 2^{k+1}}\be_{v_k}\bigl( \tau(\psib^v_{n_{k+1}}(A,x),\psib^v_{n_{k+1}}(B,y))\bigr) \\
    					&\geq \frac{1}{K(k+1)^2\cdot 2^{k+1}}\be_{v_k}\bigl(
    \tau(\psib^v_{n_{k+1}}(A,x),\psib^v_{n_{k+1}}(B,y))\bigm\vert v\in W\bigr)
    \cdot\bp(v\in W)\\[1ex]
    &\geq  \frac{1}{K(k+1)^2\cdot 2^{k+1}}\frac{(n_{k+1}-m)^d}{n_{k+1}^d}\tau((A,x),(B,y))\ ,
  \end{align*}
  where in the last inequality we used Lemma \ref{lem:folding-induced}.
  It remains to observe that 
  
 \begin{equation*}
\frac{n_{k+1}-m}{n_{k+1}}\geq \frac{2^{2^{k+1}}-2^{2^k}}{2^{2^{k+1}}}\geq \frac34,
 \end{equation*}
 and thus
 \begin{equation*}
    \norm{h(A,x)-h(B,y)}_{1}\geq  \frac{1}{K}\Bigl(\frac{3}{4}\Bigr)^d\frac{\tau((A,x),(B,y))}{\log 4\tau((A,x),(B,y))(\log\log 32\tau((A,x),(B,y)))^2}.
 \end{equation*}

We showed that there exist a constant $K'>0$ and $h\colon\Ts(\bZ^d)\to L_1$ such that for all $u,v\in\Ts(\bZ^d)$, 
  \[
  \rho\big(\tau(u,v)\big)\leq \norm{h(u)-h(v)}_{1}
\leq K'\tau(u,v),
  \]
  where $\rho(t)=\frac{t}{\log(t)(\log\log t)^2},$ and Theorem \ref{thm:coarse-embedding} follows from a straightforward modification of Lemma \ref{lem:ts-to-lamplighter}.
\end{proof}

The $(\log\log t)^2$ factor can obviously be improved. For instance by choosing the convergent series to be $\sum_{k=1}^\infty \frac{1}{k\log^{1+\e}(k+1)}$ for some $\e>0$, we could get $\rho(t)=\frac{t}{\log(t)\log\log (t) (\log\log\log t)^{1+\e}}$.

\section{A remark on Lipschitz free spaces}

Let $(X,d_X)$ be a metric space. The linear space of all finitely supported  functions $\mu\colon X\to \bR$ such that $\sum_{x\in X}\mu(x)=0$ is usually called the space of molecules over $X$. A molecule $\mu$ can be thought of as a finitely supported measure such that $\mu(X)=0$. The space of molecules can be equipped with a norm as follows:
\begin{equation}
\|\mu\|_{\LF}:=\inf \Big\{\sum_{j=1}^r |a_j|d_{X}(p_j,q_j)\colon \mu=\sum_{j=1}^ra_j(\1_{\{p_j\}}-\1_{\{q_j\}})\Big\}.
\end{equation} 
The Lipschitz free space over $X$, denoted $\LF(X)$, is the completion of the space of molecules with respect to $\|\cdot \|_{\LF}$. We refer to \cite{Weaver18} for a thorough discussion on Lipschitz free spaces (called Arens-Eells spaces there). It is convenient to think of a representation \[\mu=\sum_{j=1}^ra_j(\1_{\{p_j\}}-\1_{\{q_j\}})\] as a transportation plan, and of $\sum_{j=1}^r |a_j|d_{X}(p_j,q_j)$ as the cost of this transportation plan. It is not difficult to see that the infimum above is attained, and thus we can assume that there is a transportation cost for $\mu$ which realizes $\|\mu\|_{\LF}$. We will refer to such a plan as a \emph{minimizing plan} for $\mu$.
It is an intriguing problem to understand the relationship between the geometry of free spaces and the geometry of the Banach space $L_1$. 

In this appendix we are concerned with such connections for free spaces over finite metric spaces, and from now on we will only consider finite spaces unless explicitly mentioned otherwise. If $X$ is a finite metric space on $n+1$ points then $\LF(X)$ is a $n$-dimensional Banach space, and a natural way to quantify its proximity to $L_1$ is by studying its Banach-Mazur distortion to subspaces of $L_1$. For a Banach space $X$, we will denote by $d_Y(X)$ the $Y$-Banach-Mazur distortion of $X$, i.e. the quantity
\begin{equation}
d_Y(X)=\inf\{\|T\|\cdot\|T^{-1}\|\mid T\colon X\to Y \textrm{ into isomorphism}\}.
\end{equation}
If $Y=L_p$ we simply write $d_p(X)$ and if $Y=\ell_p^k$ we will write $d_{p,k}(X)$. If $F$ has dimension $k$ it is plain that $d_1(F)\le d_{1,k}(F)$. It is well known that if $d$ is a metric supported on a finite tree, then $\LF(T,d)$ is isometrically isomorphic to $\ell_1^{|T|-1}$ (see \cite[Proposition 2.1]{DKO20} for a proof) and thus $d_{1,|T|-1}(\LF(T,d))=1$. The Banach Mazur distortion between $\LF(X)$ and $\ell_1^{|X|-1}$, the $\ell_1$-space of the corresponding dimension was investigated, by Dilworth, Kutzarova and Ostrovskii in \cite{DKO20}. It was observed that for an expander graph $G$ on $n+1$ points, $d_{1,n}(\LF(G))=\Theta(\log n)$ and it was proved that if $G$ is an unweighted graph on $n+1$ vertices then $d_{1,n}(\LF(G))=O(n^{\frac{8}{11}})$ \cite[Proposition 3.9]{DKO20}. Also, for $D_k$ the $k$-diamond graph it was shown \cite[Proposition 3.9]{DKO20} that $d_{1,|D_k|-1}(\LF(D_k))=\Theta(\log |D_k|)$. A similar estimate was proven for Laakso graphs and multi-branching diamond, and a similar lower bound \cite[Theorem 4.2]{DKO20} was provided for a class of recursively defined graphs that generalizes the diamond graph construction. Their results improve over the best known asymptotics for $d_{1,k}(F)$ where $F$ is an arbitrary $k$-dimensional Banach space (see \cite{Tikhomirov19} (for a $\Omega(k^{\frac{5}{9}}\log^{-c} k)$ lower bound for some universal constant $c>0$) and \cite{Youssef14} (for a $(2k)^{\frac{5}{6}}$ upper bound) and the references to prior work therein). 

In this appendix we do not restrict ourself to the study of Banach-Mazur distortion of $\LF(X)$ into the $\ell_1$-space of the corresponding dimension. Recall that if $F$ is an arbitrary $n$-dimensional Banach space we have $d_1(F)\leq \sqrt{n}$. Indeed, by John's theorem $d_2(F)\leq \sqrt{n}$ and $\ell_2^n$ embeds isometrically into $L_1$. This upper bound is essentially tight since any embedding of $\ell_\infty^n$ into $L_1$ incurs Banach-Mazur distortion at least $\Omega(\sqrt{n})$. Talagrand \cite{Talagrand90} showed that a finite-dimensional subspace $F$ of $L_1$ can always be isomorphically embedded with Banach-Mazur distortion $1+\varepsilon$ into $\ell_1^k$ with $k=O(K(F)^2\varepsilon^{-2}n)$, where $K(F)$ is the K-convexity constant of $F$. Since Pisier \cite{Pisier80} proved that $K(F)=O(\sqrt{\log n})$ when $F$ is an $n$-dimensional subspace of $L_1$, it follows any $n$-dimensional Banach space $F$ isomorphically embeds into $\ell_1^{O(\varepsilon^{-2}n\log n)}$ with Banach-Mazur distortion at most $(1+\varepsilon)\sqrt{n}$. It turns out that $d_1(F)$ is significantly smaller than the general upper bound above when $F$ is the free space over a metric space $X$ such that $|X|=n+1$. 

\begin{thm}\label{theo:general}
Let $(X,d_X)$ be a finite metric space such that $|X|=n$. Then $\LF(X)$ isomorphically embeds into $\ell_1^{O(n\log n)}$ with Banach-Mazur distortion at most $O(\log n)$.
\end{thm}

Theorem \ref{theo:general} supports the evidence that a free space cannot be too far from an $L_1$-space. It also follows from Theorem \ref{theo:general} that if we are not restricted to consider $\ell_1$-spaces of the corresponding dimension, and if we are willing to pay the price of a logarithmic loss in the dimension, then the Banach-Mazur distortion for the free space over an unweighted graph on $n$ vertices given in \cite{DKO20} can be exponentially improved. The tradeoff between dimension and distortion is a very interesting problem that was already raised in \cite[Section 5, Problem 2]{NS07}.

A proof of Theorem \ref{theo:general}, which relies on the concept of stochastic embeddings into dominating tree-metrics, could be given by invoking two principles which can be extracted from \cite{NS07} and \cite{Charikar02}. Recall that $\mathsf{W}_1(X)$, the Wasserstein-1 space over a finite metric space $X$, is the set of probability measures over $X$ equipped with the $\ell_1$-transportation cost metric. In \cite[Lemma 3.1]{NS07}, Naor and Schechtman essentially proved that the isomorphic embeddability of the Banach space $\LF(X)$ into $L_1$ follows from the bi-Lipschitz embeddability of the metric space $\mathsf{W}_1(X)$. The proof that is given in \cite{NS07} deals with the special case of finite $2$-dimensional grids but the argument extends straightforwardly to any finite metric space (and to other $L_p$-spaces). The second principle asserts that if a metric space $X$ can be stochastically $D$-embedded into dominating tree-metrics then the transportation cost metric over $X$ can be embedded into $\ell_1$ with distortion $O(D)$. It was explicitly mentioned in \cite{IndykThaper03} that this principle was contained implicitly in a result of Charikar from \cite{Charikar02} whose proof was based on the techniques of Kleinberg and Tardos \cite{KleinbergTardos99,KleinbergTardos02}. 

We will prove an elementary general lemma, which makes explicit a direct link between stochastic embeddings into dominating metrics and Lipschitz free spaces, and that allows us to quote the literature on stochastic embeddings into dominating tree-metrics, namely the upper bound proved by Fakcharoenphol, Rao and Talwar in \cite{FRT04}, to obtain Theorem \ref{theo:general}. The proof given below does not need the auxiliary use of a transportation cost metric since we work directly with the Lipschitz free norm. In particular, the proof bypasses the reduction step in \cite{NS07} and thus does not require Banach spaces theory results needed in \cite{NS07} such as Ribe theorem (or differentiation techniques), nor the principle of local reflexivity. The body of work in stochastic embeddings, which is used as a black box, is obviously the heart of the matter, but its connection to Lipschitz free spaces will hopefully now appear to be elementary and somewhat more transparent. It is worth noting several other interesting consequences listed in the theorem below.

\begin{thm}\label{theo:consequences}
Let $(X,d_X)$ be a metric space of one of the following forms:
\begin{enumerate}
\item $d_X$ is a metric supported on a finite $k$-outerplanar graph $G$, for some $k\geq 1$,
\item $d_X$ is the graph metric on a finite graph $G$ with finite Euler characteristic,
\item $d_X$ is a metric supported on a finite graph with pathwidth $k$, for some $k\ge 2$,
\item $d_X$ is a metric supported on a finite graph which does not have $T$ as a minor, where $T$ can be any tree.
\end{enumerate}
Then $\LF(X)$ isomorphically embeds into $\ell_1^{O(|X|\log |X|)}$ with Banach Mazur distortion $O_X(1)$, where the $O_X(1)$ bound might depend on some structural properties of $X$ but not on the dimension of $\LF(X)$.
\end{thm}

All the statements in Theorem \ref{theo:consequences} follow from the discussion in Section 4. We now proceed with the proof of a lemma from which Theorem \ref{theo:general} will be easily deduced. It will be convenient to introduce the following definition.  

\begin{defn}
Let $\mathsf{F}$ be a class of Banach spaces. For $D\geq1$, we say that a Banach space $E$ admits a \emph{stochastic $D$-embedding into dominating $\mathsf{F}$-norms} if there exist non-negative numbers $p_1, p_2,\dots,p_k$ such that $\sum_{i=1}^k p_i=1$, and linear maps $T_i\colon E\to (F_i,\|\cdot\|_i)\in\mathsf{F}$ such that for all $x\in X$
\begin{equation}\label{eq:lower-norm}
\|T_i(x)\|_i\geq \|x\|_E \textrm{ for all } 1\leq i\leq k
\end{equation}
and 
\begin{equation}\label{eq:upper-norm}
\sum_{i=1}^k p_i\|T_i(x)\|_i\leq D\|x\|_E.
\end{equation}
\end{defn}

A simple concatenation argument gives the following lemma.  

\begin{lem}\label{lemm:stochastic-to-deterministic}
Let $\mathsf{F}$ be a collection of Banach spaces and $\sup_{F\in\mathsf{F}}d_{1}(F)\le K$. If $(E,\|\cdot\|)$ admits a stochastic $D$-embedding into dominating $\mathsf{F}$-norms then $E$ isomorphically embeds into $L_1$ with Banach-Mazur distortion at most $D\cdot K$.
\end{lem}

The next lemma is a stochastic version of a key property of Lipschitz free spaces (see for instance \cite[Lemma 2.2]{GodefroyKalton03}. If $X$ embeds bi-Lipschitzly into $Y$ with distortion at most $D$ then $\LF(X)$ isomorphically embeds in to $\LF(Y)$ with Banach-Mazur distortion at most $D$. with will allow us to reduce the $L_1$-embedding problem for $\LF(X)$ to the the problem of the stochastic embeddability of $X$ into dominating tree-metrics.
 
\begin{lem}\label{lemm:lifting}
If a metric space $(X,d_X)$ admits a stochastic $D$-embedding into dominating $\mathsf{Y}$-metrics, then $\LF(X)$ admits a stochastic $D$-embedding into dominating $\{\LF(Y)\colon Y\in\mathsf{Y}\}$-norms. 
\end{lem}

\begin{proof}
Assume that $X$ admits a stochastic $D$-embedding into dominating $\mathsf{Y}$-metrics and let $(p_i)_{i=1}^k$ and $f_i\colon X\to (Y_i,d_i)$ witness this assumption. Define $T_i\colon \LF(X)\to\LF(Y_i)$ by $T_i(\mu)(y)=\mu(f_i^{-1}(y))$ if $y\in f(X)$ and $0$ otherwise. Note that $T_i$ is well-defined since $f_i$ is injective. It is easy to verify that $T_i(\mu)\in\LF(Y_i)$ and that $T_i$ is linear. We will show that for all $\mu$ in $\LF(X)$ and for all $1\leq i\leq k$, $\|T_i(\mu)\|_i\geq \|\mu\|_X$ and $\sum_{i=1}^k p_i\|T_i(\mu)\|_i\leq D\cdot\|\mu\|_X$.

Since $f_i$ is injective and thus a bijection between $X$ and $f(X)\subset Y$, there is a bijective correspondence between the transportation plans for $\mu\in \LF(X)$ and those for $T_i(\mu)$. For instance, if we let $m_{pq}:=\1_{\{p\}}-\1_{\{q\}}$ and if $\sum_{j=1}^ra_km_{p_jq_j}$ is a transportation plan for $\mu$ then for all $y\in Y_i$, 
\begin{align*}
\sum_{j=1}^ra_j \Big(\1_{\{f_i(p_j)\}}(y)-\1_{\{f_i(q_j)\}}(y)\Big)= \sum_{j=1}^ra_j \Big(\1_{\{p_j\}}(f_i^{-1}(y))-\1_{\{q_j\}}(f_i^{-1}(y))\Big)=\mu\Big(f_i^{-1}(y)\Big)=T_i(\mu)(y),
\end{align*}
and thus
$\sum_{j=1}^r a_jm_{f_i(p_j)f_i(q_j)}$ is a transportation plan for $T_i(\mu)$. 

Let $\mu\in\LF(X)$ and let $T_i(\mu)=\sum_{j=1}^r a^{(i)}_j m_{p^{(i)}_j q^{(i)}_j}$ be a minimizing transportation plan for $T_i(\mu)$. Sicne a transportation plan for $T_i(\mu)$ is supported on $f(X)$, there exist $u^{(i)}_j, v^{(i)}_j\in X$ be such that $p^{(i)}_j=f_i(u^{(i)}_j)$ and $q^{(i)}_j=f_i(v^{(i)}_j)$. Therefore,
$$\|T_i(\mu)\|_i=\sum_{j=1}^r |a^{(i)}_j |d_{Y_i}\Big(p^{(i)}_j,q^{(i)}_j\Big)=\sum_{j=1}^r |a^{(i)}_j|d_{Y_i}\Big(f_i(u^{(i)}_j),f_i(v^{(i)}_j)\Big)\geq \sum_{j=1}^r |a^{(i)}_j|d_{X}\Big(u^{(i)}_j,v^{(i)}_j\Big)\geq \|\mu\|_X,$$

where for the last inequality we use the fact that $\sum_{j=1}^r |a^{(i)}_j|d_{X}(u^{(i)}_j,v^{(i)}_j)$ is a transportation plan for $\mu$. Indeed, for all $x\in X$, we have 
\begin{align*}
\sum_{j=1}^r a^{(i)}_j\Big(\1_{\{u^{(i)}_j\}}(x)-\1_{\{v^{(i)}_j\}}(x)\Big)&=\sum_{j=1}^r a^{(i)}_j\Big(\1_{\{f_i^{-1}(p^{(i)}_j)\}}(x)-\1_{\{f_i^{-1}(q^{(i)}_j)\}}(x)\Big)=\sum_{j=1}^r a^{(i)}_j\Big(\1_{\{p^{(i)}_j\}}(f_i(x))-\1_{\{q^{(i)}_j\}}(f_i(x))\Big)\\
		&=T_i(\mu)(f_i(x))=\mu(f_i^{-1}(f_i(x)))=\mu(x).
\end{align*}

Now, let $\sum_{j=1}^r a_j m_{p_jq_j}$ be a minimizing transportation plan for $\mu$. As we already noted above $\sum_{j=1}^r a_j m_{f_i(p_j)f_i(q_j)}$ is a transportation plan for $T_i(\mu)$ whose cost is
\begin{align*}
\sum_{i=1}^k p_i \sum_{j=1}^{r}|a_j|d_{Y_i}\Big(f_i(p_j),f_i(q_j)\Big)&= \sum_{j=1}^{r}|a_j|\sum_{i=1}^k p_i d_{Y_i}\Big(f_i(p_j),f_i(q_j)\Big)\stackrel{\eqref{eq:upper}}{\leq}  D\sum_{j=1}^{r}|a_j|d_X(p_j,q_j)\\
 & =D\|\mu\|_X \textrm{ (by minimality of the transportation plan)},
\end{align*}
and thus by taking minima on the left-hand side we have $\sum_{i=1}^k p_i\|T_i(\mu)\|_i\leq D\|\mu\|_X.$
\end{proof}

 The trees that are constructed in stochastic embeddings into dominating tree-metrics are finite weighted trees. Since we have already noticed that $d_1(\LF(T,d))=1$ for any metric on a weighted finite tree, the corollary below follows from Lemma \ref{lemm:stochastic-to-deterministic} and Lemma \ref{lemm:lifting}.

\begin{cor}\label{coro:main}
If $(X,d_X)$ admits a stochastic $D$-embedding into dominating tree-metrics then $\LF(X)$ embeds into $L_1$ with distortion at most $D$.
\end{cor}

Theorem \ref{theo:general} now follows from the main result of Fakcharoenphol, Rao and Talwar in \cite{FRT04} which says that every $n$-point metric space admits a stochastic embedding into dominating tree-metrics with distortion $O(\log n)$. Note that the construction actually gives dominating ultrametrics (the image  of $X$ in the trees lies in the set of leaves) and that only $O(n\log n)$ ultrametrics are needed. Since $\LF(U)$ is O(1)-isomorphic to $\ell_1^{|U|-1}$ if $U$ is a finite ultrametric space (cf. \cite[Corollary 2.5]{DKO20}, \cite{Dalet15} or \cite{CuthDoucha16}), it follows that $\LF(X)$ embeds into an $\ell_1$-space  of dimension $O(n^2\log n)$. It is interesting to note that this upper bound on the dimension is obtained without using Talagrand's result, albeit with a slightly weaker guarantee. Let us conclude by mentioning two examples of Lipschitz free spaces for which there are lower bounds on their $L_1$-Banach-Mazur distortion. Bourgain \cite{Bourgain86} showed that $\sup_{d\geq 1}d_1(\LF(\{0,1\}^d))=\infty$, while Khot and Naor \cite[Corollary 2]{KhotNaor06} showed the sharp lower bound $d_1(\LF(\{0,1\}^d))=\Omega(d)$, where $\{0,1\}^d$ is equipped with the Hamming metric. In \cite[Proof of Theorem 1.1]{NS07} Naor and Schechtman showed that $d_1(\LF(\llbracket n\rrbracket^2))=\Omega(\sqrt{\log n})$ where $\llbracket n\rrbracket^2:=\{0,1,\dots,n-1\}^2$ is the $2$-dimensional $n\times n$ grid.

%\bibliographystyle{alpha}\bibliography{../../0_LaTex/mybib}\end{document}

%\bibselect{../../database}
%\bibselect{database}

% \bib, bibdiv, biblist are defined by the amsrefs package.

\end{document}